\documentclass[times]{nlaauth}

\usepackage{moreverb}

\usepackage[colorlinks,bookmarksopen,bookmarksnumbered,citecolor=red,urlcolor=red]{hyperref}

\usepackage{float}
\usepackage{subfloat}
\usepackage{subfig}
\usepackage{algorithm}
\usepackage{algorithmicx}
\usepackage{algpseudocode}

\newcommand\BibTeX{{\rmfamily B\kern-.05em \textsc{i\kern-.025em b}\kern-.08em
T\kern-.1667em\lower.7ex\hbox{E}\kern-.125emX}}

\newtheorem{thm}{Theorem}[section]
\newtheorem{defn}[thm]{Definition}

\newtheorem{lem}[thm]{Lemma}

\newcommand{\be}{\begin{equation}}
\newcommand{\ee}{\end{equation}}
\newcommand{\bea}{\begin{eqnarray}}
\newcommand{\eea}{\end{eqnarray}}

\newcommand{\mR}{\mathbb{R}}

\newcommand{\diag}{\text{diag}}

\newcommand{\argminx}{\underset{x}{\operatorname{argmin\ }}}

\def\volumeyear{2014}

\begin{document}

\runningheads{X. Xiao and D. Chen}{Multiplicative Iteration for NNQP}

\title{Multiplicative Iteration for Nonnegative Quadratic Programming}

\author{X. Xiao \affil{1}\footnotemark[2] and D. Chen\affil{1}\corrauth}

\address{\affilnum{1}School of Securities and Futures, Southwestern University of Finance and Economics, Sichuan, China, 610041}

\corraddr{School of Securities and Futures, Southwestern University of Finance and Economics, Sichuan, China, 610041. E-mail: chendonghui@swufe.edu.cn}

\begin{abstract}
In many applications, it makes sense to solve the least square problems with nonnegative constraints. In this article, we present a new multiplicative iteration that monotonically decreases the value of the nonnegative quadratic programming (NNQP) objective function. This new algorithm has a simple closed form and is easily implemented on a parallel machine. We prove the global convergence of the new algorithm and apply it to solving image super-resolution and color image labelling problems. The experimental results demonstrate the effectiveness and broad applicability of the new algorithm.
\end{abstract}


\maketitle

\footnotetext[2]{E-mail: xxiao@swufe.edu.cn}

\vspace{-6pt}

\section{Introduction}
\vspace{-2pt}

Numerical problems with nonnegativity constraints on solutions are pervasive throughout science, engineering and business. These constraints usually come from physical grounds corresponding to amounts and measurements , such as solutions associated with image restoration and reconstruction~\cite{calaresg04,nmf01,nast00,zdci08} and chemical concentrations~\cite{beke04}, etc.. Nonnegativity constraints very often arise in least squares problems, i.e.  nonnegative least squares (NNLS)
\be
\argminx F(x) = \argminx ||Ax-b||_2^2 \quad \text{s.t} \quad x \geq 0.\label{eq:nnls}
\ee
The problem can be stated equivalently as the following nonnegative quadratic programing (NNQP),
\be
\argminx F(x) = \argminx \frac{1}{2} x^T Q x - x^T h  \quad \text{s.t} \quad x \geq 0.\label{eq:nnqp}
\ee
NNLS~\eqref{eq:nnls} and NNQP~\eqref{eq:nnqp} have the same unique solution. In this article, we assume $Q = A^T A \in \mathbb{R}^{n\times n}$ is a symmetric positive definite matrix, and vector $h = A^Tb\in \mathbb{R}^{n\times 1}$.

Since the first NNLS algorithm introduced by Lawson and Hanson~\cite{laha95}, researchers have developed many different techniques to solve~\eqref{eq:nnls} and~\eqref{eq:nnqp}, such as active set methods~\cite{beke04,brjo97,laha95}, interior point methods~\cite{bemamo06}, iterative approaches~\cite{kisrdh06} etc.. Among these methods, gradient projection methods are known as the most efficient methods in solving problems with simple constraints~\cite{nowr06}. In this paper, we develop a new multiplicative gradient projection algorithm for NNQP problem. Other similar research can be found in the literature~\cite{brch11,damu86,shlisale07}. 

In the paper~\cite{damu86}, the authors studied a special NNLS problem with nonnegative matrix $Q$ and vector $b$, and proposed an algorithm called the image space reconstruction algorithm (ISRA). The corresponding multiplicative iteration is
\be
x_i \leftarrow x_i \left[ \frac{h_i}{(Qx)_i} \right].\label{eq:isra}
\ee
The proof of the convergence property of ISRA can be found in~\cite{arti95,eggermont90,hahanepr09,depierro87}. More recently, Lee and Seung generalized the idea of ISRA to the problem of non-negative matrix factorization (NMF)~\cite{nmf01}. For general matrix $Q$ and vector $b$ which have both negative and positive entries, the authors proposed another multiplicative iteration~\cite{shlisale07}
\[
x_i \leftarrow x_i \left[ \frac{h_i + \sqrt{h_i^2+4(Q^+x)_i(Q^-x)_i }}{2(Q^+x)_i  } \right].
\]
In~\cite{brch11}, Brand and Chen also introduced a multiplicative iteration
\[
x_i \leftarrow x_i \left[ \frac{(Q^-x)_i + h_i^+ }{(Q^+x)_i + h_i^- } \right],
\]
where $Q^+ = \max(Q,0)$, $Q^- = \max(-Q, 0)$, $h^+=\max(h,0)$, $h^-=\max(-h, 0)$, ``max"  is element-wise comparison of two matrices or vectors. Both above algorithms are proved monotonically converging to global minimum of NNQP objective function~\eqref{eq:nnqp}.

In this paper, we present a new iterative NNLS algorithm along with its convergence analysis. We prove that the quality of the approximation improves monotonically, and the iteration is guaranteed to converge to the global optimal solution. The focus of this paper is theoretical proof of the monotone convergence of the new algorithm. We leave the comparison with other NNLS algorithms for future research. The remainder of this paper is organized as follows. Section~\ref{nsec:alg} presents the new multiplicative NNLS algorithm, we prove the algorithm monotonically decrease the NNQP objective function. In section~\ref{nsec:exp}, we discuss two applications of the new algorithm to image processing problems, including image super-resolution and color image labelling. Finally, in section~\ref{nsec:con}, we conclude by summarizing the main advantage of our approach..

\vspace{-6pt}

\section{Multiplicative Iteration and Its Convergence Analysis~\label{nsec:alg}}

In this section we derive the multiplicative iteration and discuss its convergence properties. 
Consider the NNQP problem~\eqref{eq:nnqp}
\[
\argminx F(x) = \argminx \frac{1}{2} x^TQx - x^Th \quad s.t. \quad x \geq 0,
\]
where $Q = A^TA\in \mathcal{R}^{n\times n}$ is positive definite, $h = A^Tb \in \mR^n$.

\subsection{Multiplicative Iteration}

The proposed multiplicative iteration for solving~\eqref{eq:nnqp} is
\be
x_i \leftarrow x_i \left[ \frac{2(Q^-x)_i + h_i^+ +\delta }{(|Q|x)_i + h_i^- +\delta } \right], \label{eq:update}
\ee
where $|Q| = \text{abs}(Q) = Q^+ + Q^-$, constant stablizer $0 < \delta \ll 1 $ guarantees the iteration monotonically convergent. We will discuss how to choose $\delta$ later in this section. If all the entries in $Q$ and $h$ are nonnegative, the multiplicative update~\eqref{eq:update} reduced to ISRA~\cite{damu86}.

Since all the components of the multiplicative factors, i.e. matrices $Q^+$, $Q^-$, $|Q|$, and vectors $h^+$, $h^-$, are nonnegative, 
\[
\frac{2(Q^-x)_i + h_i^+  +\delta}{(|Q|x)_i + h_i^- +\delta},
\]
given a nonnegative starting initial guess, $x^0$, all the generated iterations, $\{x^k\}$, are nonnegative. In generating the sequence $\{x^k\}$, the iteration computes the new update $x^{i+1}$ by using only the previous vector $x^i$. It does not need to know all the previous updates, $\{x^k\}$. And the major computational tasks to be performed at each iteration are computations of the matrix-vector products, $Q^-x$ and $|Q|x$. These remarkable properties imply that the multiplicative iteration requires little storage and computation for each iteration.

The iteration~\eqref{eq:update} is a gradient projection method which can be shown by
\bea
x_i^{k+1} - x_i^k  &=& \left[ \frac{2(Q^-x^k)_i + h_i^+ +\delta }{(|Q|x^k)_i + h_i^-  +\delta} \right]x_i^k  - x_i^k \nonumber \\
&=& \left[ \frac{2(Q^-x^k)_i + h_i^+ -( |Q|x^k)_i-h_i^-}{(|Q|x^k)_i + h_i^-  +\delta} \right] x_i^k \nonumber \\
&=& -\left[ \frac{(Qx^k)_i - h_i}{(|Q|x^k)_i + h_i^-  +\delta} \right] x_i^k \nonumber \\
&=& -\left[ \frac{x^k_i}{(|Q|x^k)_i + h_i^-  +\delta} \right] ((Qx^k)_i - h_i) \nonumber\\
& =& - \gamma_k \nabla(F(x^k)), \nonumber
\eea
where the step-size $\gamma_k = \left[ \frac{x^k_i}{(|Q|x^k)_i + h_i^-  +\delta} \right]$, and the gradient of the NNQP objection function~\eqref{eq:nnqp} $ \nabla(F(x)) = Qx - h$.

\subsection{Fixed Point}

The proposed iteration~\eqref{eq:update} is motivated by the Karush-Kuhn-Tucker (KKT)  first-order optimal condition~\cite{nowr06}. Consider the Lagrangian function of NNQP objective function~\eqref{eq:nnqp} defined by
\be
\mathcal{L}(x,\mu) = \frac{1}{2} x^T Q x - x^T h - \mu x, \label{eq:nnqplag}
\ee
with the scalar Lagrangian multiplier $\mu_i \geq 0$, assume $x^*$ is the optimal solution of Lagrangian function~\eqref{eq:nnqplag}, the KKT conditions are
\bea
x^* \circ (Qx^* - h - \mu) &=& 0 \nonumber \\
\mu \circ x^* &=& 0,\nonumber
\eea
with $\circ$ denoting the Hadamard product. Above two equalities imply that either $i$th constraint is active $x^*_i=0$, or $\mu_i = 0$ and $(Qx^*)_i - h_i = 0$ when the $i$th constraint is inactive ($x^*_i>0$). Because $Q = |Q| - 2Q^-$, equality $(Qx^*)_i - h_i = 0$ implies $((|Q|-2Q^-)x^*)_i - (h^+_i - h_i^-) = 0$, we obtain $((|Q|)x^*)_i + h_i^-  +\delta = 2(Q^-x^*)_i + h_i^+ +\delta$, which is equivalent to 
\[
\frac{2(Q^-x^*)_i + h_i^+ +\delta}{(|Q|x^*)_i + h_i^- +\delta} = 1,
\]
which means the $i$th multiplicative factor is constant $1$.Therefore, any optimal solution $x^*$ satisfying the KKT conditions conrresponds to a fixed point of the multiplicative iteration.

\subsection{Convergence Analysis}

In this section, we prove the proposed multiplicative iteration~\eqref{eq:update} monotonically decrease the value of the NNQP objective function~\eqref{eq:nnqp} to its global minimum. This analysis is based on construction of an auxiliary function of $F(x)$. Similar techniques have been used in the papers~\cite{nmf01,shlisale07,shsale02}. 

For the sake of completeness, we begin our discussion with a brief review of the definition of auxiliary function.
\begin{defn}\label{def:aux}
Let $x$ and $y$ be two positive vectors, function $G(x,y)$ is an auxiliary function of $F(x)$ if it satisfies the following two properties
\begin{itemize}
\item $F(x) < G(x,y) \quad \mbox{if} \quad x\neq y$;
\item $F(x)=G(x,x)$.
\end{itemize}
\end{defn}

\begin{figure}[t]
\centering
\includegraphics[width= 0.8\textwidth]{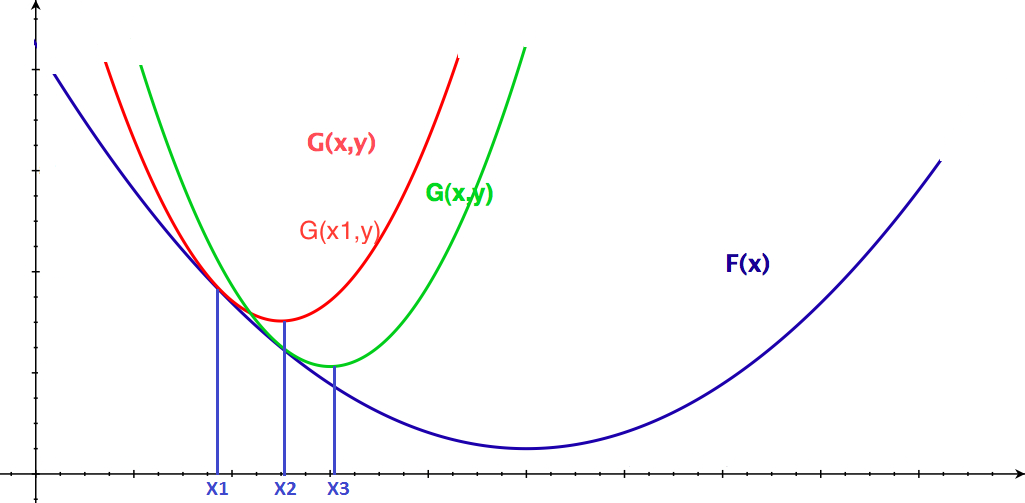}
\caption[Auxiliary function]{Graph illustrating how to searching the minimum of the objective function $F(x)$ by minimizing the auxiliary function $G(x,y)$ iteratively. Each new update $x^i$ is computed by  searching the minimum of the auxiliary function $G(x,y)$ in every step.}
\label{fig:auxfn}
\end{figure}

Figure~\ref{fig:auxfn} illustrates the relationship between the auxiliary function $G(x,y)$ and the corresponding objective function $F(x)$. In each iteration, the updated $x^i$ is computed by minimizing the auxiliary function. The iteration stops when it reaches a stationary point. The following lemma, which is also presented in~\cite{nmf01,shlisale07,shsale02}, proves the iteration by minimizing auxiliary function $G(x,y)$ in each step decreases the value of the objective function $F(x)$.
\begin{lem}\label{lem:aux}
Let $G(x,y)$ be an auxiliary function of $F(x)$, then $F(x)$ is strictly decreasing under the update
\[
x^{k+1} = \argminx G(x,x^k),
\]
if $x^{k+1} \neq x^{k}$.
\end{lem}

\begin{proof}
By the definition of auxiliary function, if $x^{k+1} \neq x^k$, we have
\[
F(x^{k+1}) < G(x^{k+1}, x^k) \leq G(x^k,x^k) = F(x^k).
\]
The middle inequality is because of the assumption $x^{k+1} = \argminx G(x,x^k)$.
\end{proof}

Deriving a suitable auxiliary function for NNQP objective function $F(x)$~\eqref{eq:nnqp} is a key step to prove the convergence of our multiplicative iteration. In the following lemma, we prove two positive semi-definite matrices which are used to build our auxiliary function later.
\begin{lem}\label{lem:spd}
Let $P\in \mathbb{R}^{n\times n}$ be a nonnegative real symmetric matrix without all-zero rows and $x\in \mathbb{R}^{n\times 1}$ be a vector whose entries are positive. Define the diagonal matrix $D\in \mathbb{R}^{n\times n}$,
\[
  D_{ij} = \left\{
  \begin{array}{l l}
    0 & \quad \text{if $i\neq j$}\\
    \frac{(Px)_i}{x_i} & \quad \text{otherwise}\\
  \end{array} \right.
\]
Then, the matrices, $(D\pm P)$, are positive semi-definite. 
\end{lem}
\begin{proof}
Consider the matrices,
\[
M_1 = \diag(x_i) (D +P) \diag(x_i), \quad M_2 = \diag(x_i) (D-P) \diag(x_i),
\]
where $\diag(x_i)$ represents the diagonal matrix whose entries on the main diagonal are the entries of vector $x$. Since $x$ is a positive vector, $\diag(x_i)$ is invertible. Hence, $D\pm P$ are congruent with $M_1$, $M_2$, correspondingly. The matrices $D\pm P$ are positive semi-definite if and only if $M_1$ and $M_2$ are positive semi-definite~\cite{hojo90}. 

 Given any nonzero vector $z\in \mathbb{R}^{n\times 1}$,
\begin{eqnarray}
z^T M_1 z &=& \sum_{ij} (D_{ij} + P_{ij})x_i x_j z_i z_j \nonumber \\
&=& \sum_{ij}D_{ij}x_i x_j z_i z_j + \sum_{ij} P_{ij} x_i x_j z_i z_j \nonumber \\
&=& \sum_i D_{ii}x_i^2z_i^2+ \sum_{ij} P_{ij} x_i x_j z_i z_j \nonumber \\
&=& \sum_i (Px)_ix_iz_i^2+ \sum_{ij} P_{ij} x_i x_j z_i z_j \nonumber \\
&=& \sum_{ij} P_{ij} x_i x_j z_i^2 + \sum_{ij} P_{ij} x_i x_j z_i z_j \nonumber \\
&=& \frac{1}{2} \sum_{ij} P_{ij} x_i x_j (z_i + z_j)^2  \geq 0 \nonumber
\end{eqnarray}
Hence, $M_1$ is positive semi-definite. Similarly, we have $M_2$ is positive semi-definite,
\[
z^T M_2 z =  \frac{1}{2} \sum_{ij} P_{ij} x_i x_j (z_i - z_j)^2  \geq 0.
\]
Therefore, $D \pm P$ are positive semi-definite.
\end{proof}

Combining two previous lemmas, we construct an auxiliary function for NNQP~\eqref{eq:nnqp} as follows.
\begin{lem}[Auxiliary Function]\label{lem:G(x,y)}
Let vectors $x$ and $y$ represent two positive vectors, define the diagonal matrix, $D(y)$, with diagonal entries
\[
D_{ii} =  \frac{(|Q|y)_i+h^-_i   +\delta}{y_i} , \quad i = 1, 2, \cdots, n, \quad  \delta>0.
\]
Then the function
\be
G(x,y) = F(y)+(x-y)^T\nabla F(y)+\frac{1}{2}(x-y)^TD(y)(x-y) \label{eq:aux}
\ee
is an auxiliary function for quadratic model
\[
F(x) = \frac{1}{2}x^TQx -x^Th.
\]
\end{lem}

\begin{proof}
First of all, it is obvious that $G(x,x) = F(x)$ which is the second property in Definition~\ref{def:aux}. Next, we have to show the first property, $G(x,y)> F(x)$ for $x\neq y$. 

Notice that $Q$ is the Hesssian matrix of $F(x)$. The Taylor expansion of $F(x)$ at $y$ is
\[
F(x) = F(y)+(x-y)^T\nabla F(y) +\frac{1}{2}(x-y)^TQ(x-y).
\]
The difference between $G(x,y)$ and $F(x)$ is 
\[
G(x,y) - F(x) = \frac{1}{2}(x-y)^T(D(y)-Q)(x-y).
\]
$G(x,y)> F(x)$ for $x\neq y$ if and only if $(D(y)-Q)$ is positive definite.

Recall that $|Q| = Q^+ + Q^-$, $|Q|y = Q^+y + Q^-y$,
\begin{eqnarray}
D(y)-Q &=& \diag\left(\frac{(|Q|y)_i+h_i^- +\delta }{y_i}\right)-Q \nonumber \\
&=&  \diag\left(\frac{(|Q|y)_i+h_i^- +\delta }{y_i}\right) - ( Q^+ -Q^- )\nonumber \\
&=&  \diag\left(\frac{(Q^+y)_i}{y_i}\right)-Q^+ + \diag\left(\frac{(Q^-y)_i}{y_i}\right)+Q^- +\diag\left(\frac{h^-_i +\delta}{y_i}\right)\nonumber
\end{eqnarray}
Because $\diag\left(\frac{h^-_i +\delta}{y_i}\right)$ is positive definite for $\delta > 0$, and by Lemma~\ref{lem:spd}, $ \diag\left(\frac{(Q^+y)_i}{y_i}\right)-Q^+$ and $\diag\left(\frac{(Q^-y)_i}{y_i}\right)+Q^-$ are positive semi-definite. Thus, $(D(y)-Q)$ is positive definite. 

Hence, we obtain $G(x,y)>F(x)$ for any vectors $x\neq y$. Therefore, $G(x,y)$ is an auxiliary of $F(x)$.
\end{proof}

In previous proof, we use the fact $\delta > 0$ to prove the diagonal matrix $\diag(\frac{h_i^- + \delta}{y_i})$ is positive definite. Because matrix $\diag(\frac{h_i^- }{y_i})$ is positive semi-definite for vector $y$ with all positive entries, we can choose $\delta$ to be any positive number. In our experiments, it is chosen to be $eps = 10^{-16}$. Armed with previous lemmas, we are ready to prove the convergence theorem for our multiplicative iteration~\eqref{eq:update}.

\begin{thm}[Monotone Convergence]\label{thm:nnqp}
The value of the objective function $F(x)$ in~\eqref{eq:nnqp} is monotonically decreasing under the multiplicative update
\[
x^{k+1}_i = x^k_i \left[ \frac{2(Q^-x^k)_i+h_i^+ +\delta }{(|Q|x^k)_i+h_i^-  +\delta}\right].
\]
It attains the global minimum of $F(x)$ at the stationary point of the iteration.
\end{thm}

\begin{proof}
By the auxiliary function definition~\ref{def:aux} and Lemma~\ref{lem:G(x,y)}, $G(x,y)$ in~\eqref{eq:aux} is an auxialiry of function $F(x)$. Lemma~\ref{lem:aux} shows that the objective function $F(x)$ is monotonically decreasing under the update
\[
x^{k+1} = \argminx\  G(x,x^k) \quad \mbox{if}\quad x^{k+1} \neq x^k.
\]
It remains to show that the proposed iteration~\eqref{eq:update} approaches the minimum of $G(x,x^k)$. 

By Fermat's theorem~\cite{rudin}, taking the first partial derivative of $G(x,y)$ with respect to $x$, and setting it to $0$, we obtain that
\begin{equation}
\nabla_x G(x,x^k) = \nabla F(x^k) + D(x^k)(x-x^k) = 0.
\end{equation}
Hence, 
\begin{eqnarray}
x &=& x^k - (D(x^k))^{-1}\nabla F(x^k) \nonumber\\
&=& x^k - (D(x^k))^{-1} (Qx^k-h) \nonumber \\
&=& x^k - (D(x^k))^{-1} (|Q|x^k+h^- +\delta -2Q^-x^k-h^+-\delta) \nonumber \\
&=& x^k -  (D(x^k))^{-1} (|Q|x^k+h^- +\delta) + (D(x^k))^{-1} (2Q^-x^k+h^+ +\delta)\nonumber \\
&=&  (D(x^k))^{-1} (2Q^-x^k+h^+ +\delta)\nonumber \\
&=& \diag \left( \frac{2(Q^-x^k)_i+h_i^+ +\delta} {(|Q|x^k)_i+h^-_i +\delta}\right) x^k \nonumber
\end{eqnarray}
where we used the facts that  $(D(x^k))^{-1} (|Q|x^k+h^- +\delta) =x^k$.

The decreasing sequence $\{F(x^k)\}$ is bounded below $-b^Tb$. By Monotone Convergence Theorem~\cite{rudin}, the sequence converges to the limit $F^*$. Because $F(x)$ is continuous, given any compact domain, there exists $x^*$ such that $F(x^*) = F^*$. Since $F(x^*)$ is the global minimum of $F(x)$, the gradient of $F(x)$ at $x^*$ is zero, i.e.
\[
\nabla F(x^*) = Qx^* - h = 0,
\]
which is equivalent to 
\[
 \frac{2(Q^-x^*)_i+h_i^+ +\delta}{(|Q|x^*)_i+h_i^- +\delta} =1,
\]
which means $x^*$ is a stationary point. Thus, the sequence $\{F(x^k)\}$ converges to the global minimum $F(x^*)$ as $\{x^k\}$ approaches to the limit point $x^*$.
\end{proof}

\vspace{-6pt}

\section{Numerical Experiments~\label{nsec:exp}}

We now illustrate two applications of the proposed NNLS algorithm in image processing problems.

\subsection{Image Super-Resolution}

 \begin{figure}[t]
\centering
\subfloat[4 frames of 30 input low-resolution frames]{\label{fig:textlr1}\includegraphics[width= 0.4\textwidth]{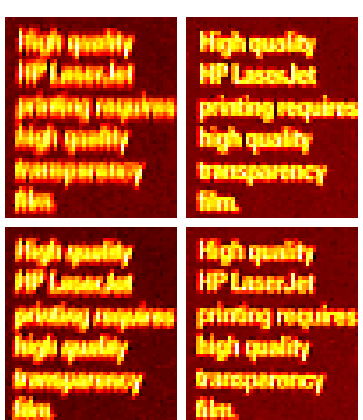}} \hspace{.25in}
\subfloat[Restored high-resolution image by the proposed multiplicative iteration]{\includegraphics[width= 0.385\textwidth]{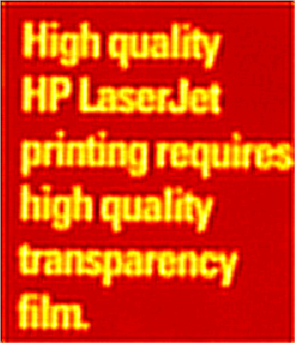}\label{fig:textSR4}}
\caption[]{Super-resolution example 1. Left: 4 sample of 30 input low-resolution image with size of $57\times 49$ pixels. Right: restored high-resolution image with size $285\times 245$ pixels.}
\label{fig:textSR}

\subfloat[4 frames of 16 input low-resolution frames]{\label{fig:eialr1}\includegraphics[width= 0.4\textwidth]{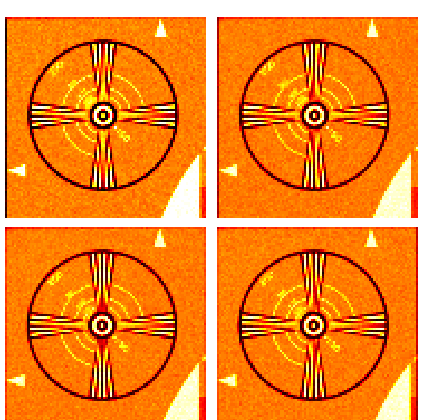}} \hspace{.25in}
\subfloat[Restored high-resolution image by the proposed multiplicative iteration]{\includegraphics[width= 0.382\textwidth]{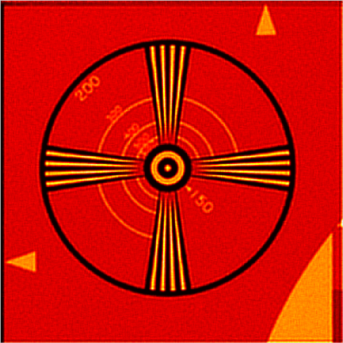}\label{fig:eiaSR2}}
\caption[]{Super-resolution example 1. Left: 4 sample of 16 input low-resolution image with size of $90\times 90$ pixels. Right: restored high-resolution image with size $360\times 360$ pixels.}
\label{fig:eiaSR}
\end{figure}

Image super-resolution (SR) refers to the process of combining a set of low resolution images into a single high-resolution image~\cite{chen08,chhana06}. Each low-resolution image $y_k$ is assumed to be generated from an ideal high-resolution image $x$ via a displacement $S_k$, a decimation $D_k$, and a noise process $n_k$:
\be
y_k = D_kS_kx + n_k, \quad k=1,2,\cdots, K. ~\label{eq:sr}
\ee

We use the bilinear interpolation proposed by Chung, Haber and Nagy~\cite{chhana06}~\footnote{Thanks to Julianne Chung for providing the Matlab code.} to estimate the displacement matrices $S_k$. Then we
reconstruct the high-resolution image by iteratively solving the NNLS
\[
\argminx \frac{1}{2}\sum_{k=1}^K ||D_k S_k x - y_k||^2, \quad \mbox{s.t.} \quad x\geq 0
\]

The low-resolution test data set is taken from the Multi-Dimensional Signal Processing Research Group (MDSP)~\cite{faroelmi03}. Figure~\ref{fig:textlr1} shows 4 of the $30$ uncompressed low-resolution text images of size $57\times 49$ pixels. The reconstructed high-resolution image of size $285\times 245$ pixels is computed by our algorithm is shown in Figure~\ref{fig:textSR4}. Figure~\ref{fig:eialr1} shows 4 of $16$ low-resolution EIA images of size $90\times 90$ pixels. Figure~\ref{fig:eiaSR2} shows the reconstructed $360\times 360$ pixels high-resolution image computed by the proposed multiplicative iteration. As shown in the figures, the high-resolution images are visually much better than the low-resolution images.

\subsection{Color Image Labeling}

In Markov random fields (MRF)-based interactive image segmentation techniques, the user labels a small subset of pixels, and the MRF propagates these labels across the image, typically finding high-gradient contours as segmentation boundaries where the labeling changes~\cite{ridata08}. These techniques require users to impose hard constraints by indicating certain pixels (seeds) that absolutely have to be part of the labeling $k$. Intuitively, the hard constraints provide clues as to what the user intends to segment. Denote $X$ as the $m$-by-$n$ test RGB images, $X_{ij}$ represent a $3$-by-$1$ vector at pixel $(i,j)$.

 \begin{figure}[t]
\centering
\subfloat[flowers]{\label{fig:flower}\includegraphics[width= 0.4\textwidth]{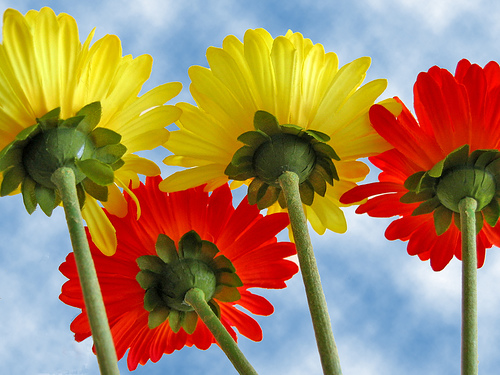}} \hspace{.25in}
\subfloat[segmented image]{\includegraphics[width= 0.4\textwidth]{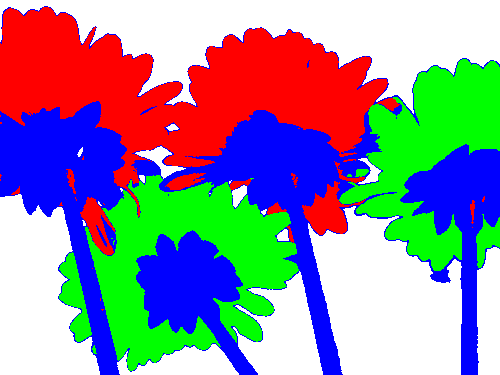}\label{fig:flowerseg}}\\
\subfloat[Manhattan skyline]{\label{fig:man}\includegraphics[width= 0.4\textwidth]{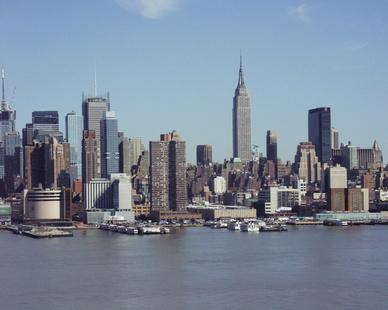}} \hspace{.25in}
\subfloat[segmented image]{\includegraphics[width= 0.4\textwidth]{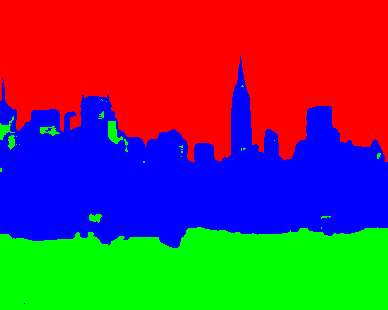}\label{fig:manseg}}
\caption[NNLS/MRF image labeling results]{Sample image labeling results. using MRF model solved by the proposed NNLS algorithm}
\label{fig:label}
\end{figure}

The class set is denoted by $\mathcal{C}=\{1,2,\cdots,K\}$. The probabilistic labeling approaches compute a probability measure field for each pixel $(i,j)$,
\[
\mathcal{X} = \{X^k_{ij}: k\in\mathcal{C}, i=1,2,\cdots,m, j = 1,2,\cdots,n\}
\]
with the constraints
\be
\sum_{k=1}^KX^k_{ij} =1, \quad X^k_{ij}\geq 0, \quad \forall k\in \mathcal{C}.~\label{eq:segcon}
\ee
Denoting $\mathcal{N}_{ij} = \{(i', j'): \min\{ |i'-i|,|j'-j| \} = 1\}$ as the set of neighbors of pixel $(i,j)$, the cost function are in the following quadratic form
\be
\argminx \sum_{k=1}^K \sum_{i=1}^m \sum_{j=1}^n \left(\frac{\alpha}{2}\sum_{(i',j')\in\mathcal{N}_ij}\omega_{iji'j'}(X^k_{i'j'}-X^k_{ij})^2 + D^k_{ij} X^k_{ij}\right),~\label{eq:segobj}
\ee
with the constraints~\eqref{eq:segcon}. $D^k_{ij}$ is the cost of assigning label $k$ to pixel $(i,j)$. The first term, $\sum_{(i',j')\in\mathcal{N}_ij}\omega_{iji'j'}(X^k_{i'j'}-X^k_{ij})^2$, which controls the granularity of the regions, promotes smooth regions. The spatial smoothness is controlled by the positive parameter, $\alpha$, and weight, $\omega$, which is chosen such that $\omega_{iji'j'}\approx 1$ if the neighbouring pixels $(i,j)$ and $(i',j')$ are likely to belong to the same class and $\omega_{iji'j'}\approx 0$ otherwise. In these experiments, $\omega$ is defined to be the cosine of the angle between two neighbouring pixels,
\[
\omega_{iji'j'} = \frac{X_{ij}^T X_{i'j'}}{|X_{ij}|\cdot |X_{i'j'}|}.
\]

\begin{algorithm}[t]
\caption{NNLS Algorithm MRF Image Segmentation}\label{alg:nnlsseg}
\begin{algorithmic}[1]
\While{$\text{norm}((X^k)^\text{new}- (X^k)^\text{old})/\text{norm}((X^k)^\text{old}) > \text{tol}$}
\State Update the probability measure field
\[
(X^k_{ij})^\text{new} = (X^k_{ij})^\text{old} * \frac{2\alpha \sum_{(i',j')\in\mathcal{N}_{ij}}\omega_{iji'j'}X^k_{ij} + (D^k_{ij})^-+\lambda_{ij}} {\alpha  X^k_{ij}\left(\sum_{(i',j')\in\mathcal{N}_{ij}}\omega_{iji'j'}\right) + \alpha\sum_{(i',j')\in\mathcal{N}_{ij}}\omega_{iji'j'}X^k_{ij} + (D^k_{ij})^+}
\]
\State Update the Lagrangian parameter
\[
\lambda_{ij}^\text{new} = \lambda_{ij}^\text{old}* \frac{1}{\sum_{k}X^k_{ij}}
\]
\EndWhile
\State \textbf{return} $(X^k)^\text{new}$
\end{algorithmic}
\end{algorithm}

The cost of labeling $k$ at each pixel $(i,j)$, $D^k_{ij}$, is trained with a Gaussian mixture model~\cite{xujo95} using seeds labeled by the user. Given sample mean, $\mu^{k}$, and variance, $\sigma^{k}$, for the seeds with labeling $k$, $D^k_{ij}$ is computed as the Mahalanobis distance~\cite{ma36} between each pixel of the image and the seeds,
\[
D^k_{ij} = \frac{1}{2}\sum_{k=1}^K (X_{ij}-\mu^{k})^T(\Sigma^{k})^{-1}(X_{ij}-\mu^{k}) + \frac{1}{2}\log(\Sigma^{k}).
\]
The KKT optimality conditions
\bea
\alpha\sum_{(i',j')\in\mathcal{N}_{ij}}\omega_{iji'j'}(X^k_{ij}-X^k_{i'j'}) + D^k_{ij} - \lambda_{ij} &=& 0 \nonumber \\
\alpha X^k_{ij}\left(\sum_{(i',j')\in\mathcal{N}_{ij}}\omega_{iji'j'}\right)-\alpha \sum_{(i',j')\in\mathcal{N}_{ij}}\omega_{iji'j'}X^k_{ij} + D^k_{ij}-\lambda_{ij} &=& 0\nonumber
\eea
yield a two-step Algorithm~\ref{alg:nnlsseg}.

In the experiments, we implement our NNLS algorithm without explicitly constructing the matrix $Q$. Figure~\ref{fig:label} shows the results of the labeled images. 

\vspace{-6pt}

\section{Summary~\label{nsec:con}}

In this paper, we presented a new graduent projection NNLS algorithm and its convergence analysis. By expressing the KKT first-order optimality conditions as a ratio, we obtained a multiplicative iteration that could quickly solves large quadratic programs on low-cost parallel compute devices. The iteration monotonically converges from any positive initial guess. We demonstrated applications to image super-resolution and color image labelling. Our algorithm is also applicable to solve other optimization problems involving nonnegativity contraints. Future research includes comparing the performance of this new NNLS algorithm with other existing NNLS algorithms.

\bibliographystyle{plain}
\bibliography{nnls_ref}

\end{document}